\documentclass[12pt]{amsart}
\usepackage{geometry}                
\usepackage{graphicx}
\usepackage{amssymb}
\usepackage{epstopdf}

\usepackage{amsfonts}
\usepackage{amssymb}
\usepackage{amscd}
\usepackage{amsbsy}
\usepackage{amsthm}
\usepackage{graphicx}
\usepackage{mathrsfs} 
\usepackage{txfonts} 

\def\N{{\mathbb N}} 
\def\Z{{\mathbb Z}}

\def\R{{\mathbb R}}
\def\C{{\mathbb C}}

\def\Im{\mathrm{Im}\,}
\def\Re{\mathrm{Re}\,}
\def\supp{\mathrm{supp}\,}

\newcommand{\cO}{\mathcal{O}}

\newcommand{\cF}{\mathcal{F}}

\newcommand{\mcK}{\mathcal{K}}
\newcommand{\mcH}{\mathcal{H}}
\newcommand{\be}{\begin{equation}}
\newcommand{\ee}{\end{equation}}
\def\hB{\widehat{B}}

\def\hL{\widehat{L}}

\newcommand{\inn}[2]{\langle #1,#2 \rangle}

\newcommand{\bfa}{\boldsymbol{a}}

\newcommand{\mydot}{\,\cdot\,}
\newcommand{\nl}{\vskip 5pt\noindent}

\DeclareMathOperator{\re}{Re}
\DeclareMathOperator{\Arg}{Arg}
\newcommand{\wh}[1]{{\widehat{#1}}}
\newcommand{\abs}[1]{{\left\vert\, #1\,\right\vert}}

\newtheorem{theorem}{Theorem}
\newtheorem{lemma}[theorem]{Lemma}

\theoremstyle{definition}
\newtheorem{definition}[theorem]{Definition}
\newtheorem{example}[theorem]{Example}
\newtheorem{proposition}[theorem]{Proposition}

\theoremstyle{remark}
\newtheorem{remark}[theorem]{Remark}

\title{Interpolation and Sampling with Exponential Splines of Real Order}

\author{Peter Massopust}
\address{Centre of Mathematics, Technical University of Munich, Germany}
\email{massopust@ma.tum.de}

\begin{document}

\maketitle

\begin{abstract}
The existence of fundamental cardinal exponential B-splines of positive real order $\sigma$ is established subject to two conditions on $\sigma$ and their construction is implemented. A sampling result for these fundamental cardinal exponential B-splines is also presented.
\end{abstract}
\section{Introduction}
Cardinal exponential B-splines of order $n\in \N$ are defined as $n$-fold convolution products of exponential functions of the form $e^{a (\cdot)}$ restricted to the unit interval $[0,1].$ More precisely, let $n\in \N$ and $\boldsymbol{a}:=(a_1, \ldots, a_n)\in \R^n$, with at least one $a_j\neq 0$, $j = 1, \ldots, n$. A {\em cardinal exponential B-spline of order $n$ associated with the $n$-tuple of parameters $\bfa$} is defined by \be\label{regE}
E_n^{\bfa} :=\underset{j = 1}{\overset{n}{*}} \left(e^{a_j (\mydot)}\chi\right),
\ee
where $\chi$ denotes the characteristic function of the unit interval $[0,1]$.

This wider class of splines shares several properties with the classical Schoenberg polynomial B-splines, but there are also significant differences that makes them useful for different purposes. In \cite{CM}, an explicit formula for these functions was established and those cases characterized for which the integer translates of an exponential B-spline form a partition of unity up to a nonzero multiplicative factor. In addition, series expansions for $L^2(\R)$--functions in terms of shifted and modulated versions of exponential B-splines were derived and dual pairs of Gabor frames based on exponential B-splines constructed. We remark that exponential B-splines have also been employed to construct multiresolution analyses and to obtain wavelet expansions. (See, e.g., \cite{LY,unserblu00}.) Furthermore, in \cite{CS} it is shown that exponential splines play an important role in setting up a one-to-one correspondence between dual pairs of Gabor frames and dual pairs of wavelet frames. For an application to some numerical methods, we refer the interested reader to \cite{EKD} and \cite{KED}.
\nl
In \cite{m14}, a new class of more general cardinal exponential B-splines, so-called cardinal exponential B-splines of complex order, was introduced, some properties derived and connections to fractional differential operators and sampling theory exhibited.

Classical polynomial B-splines $B$ can be used to derive fundamental splines which are linear combinations of integer-translate of $B$ and which interpolate the set of data points $\{(m, \delta_{m,0}) : m\in \Z\}$. As it turns out, even generalizations of these polynomial B-splines, namely, polynomial B-splines of complex and even quaternionic order, do possess associated fundamental splines provided the order is chosen to lie in certain nonempty subregions of the complex plane or quaternionic space. For details, we refer the interested reader to \cite{fgms} in the former case and to \cite{hm} in the latter.

In this article, we consider cardinal exponential B-splines of positive real order, so-called cardinal fractional exponential B-splines (to follow the terminology already in place for the polynomial B-splines). As we only deal with cardinal splines, we drop the adjective ``cardinal" from now on. By extending the integral order $n\in\N$ of the classical exponential B-splines to real orders $\sigma > 1$, one achieves a higher degree of regularity at the knots.

The structure of this paper is as follows. In Section 2 we define fractional exponential splines and present those properties that are important for the remainder of this article. The fundamental exponential B-spline is constructed in Section 3 following the procedure for the polynomial splines. However, as the Fourier transform of an exponential B-spline includes an additional positive term, the construction and the proof of existence of fundamental exponential B-splines associated with fractional exponential B-splines is more involved. Section 4 deals with a sampling result for fundamental exponential B-splines.
\section{Fractional Exponential B-Splines}\label{sec2}
In order to extend the classical exponential B-splines to incorporate real orders $\sigma$, we work in the Fourier domain. To this end, we take the Fourier transform of an exponential function of the form $e^{ -a x}\chi$, $a\in \R$, and define a \emph{fractional exponential B-spline} in the Fourier domain by
\be\label{expspline}
\widehat{E^\sigma_a} (\xi):= \cF(\xi) := \int_\R e^{ -a x}\chi(x)\,e^{- i x \xi}\,dx = \left(\frac{1-e^{-(a+i\xi)}}{a+i\xi}\right)^\sigma,\quad\xi\in \R.
\ee
Note that we may interpret the above Fourier transform for real-valued argument $\xi$ as a Fourier transform for complex-valued argument by setting $z:= \xi + i\,a$:
\be\label{complexF}
\widehat{E^\sigma_a} (\xi) = \cF(z) := \int_\R \chi(x)\,e^{- z x}\,dx,\quad z\in \C.
\ee
It can be shown \cite{m14} (for complex $\sigma$) that the function
\[
\Xi(\xi, a) := \frac{1-e^{-(a+i\xi)}}{a+i\xi},
\]
is only well-defined for $a \geq 0$. As $a=0$ yields fractional polynomial B-splines, we assume henceforth that $a > 0$.

From \cite{m14}, we immediately derive the time domain representation for a fractional exponential B-spline $E^a_\sigma$ assuming $\sigma > 1$:
\be\label{timerep}
E_a^\sigma (x) = \frac{1}{\Gamma(\sigma)}\,\sum_{k=0}^\infty \binom{\sigma}{k} (-1)^k e^{-k a} e_+^{-a(x-k)}\,(x-k)_+^{\sigma-1},
\ee
where $e_+^{(\cdot)} := \chi_{[0,\infty)}\,e^{(\cdot)}$ and $x_+ := \max\{x,0\}$. It was shown that the sum converges both point-wise in $\R$ and in the $L^2$--sense.

Next, we summarize some additional properties of exponential B-splines. 
\begin{proposition}\label{prop1}
$\abs{\widehat{E^\sigma_a}}\in \cO(\abs{\xi}^{-\sigma})$ as $\abs{\xi}\to\infty$.
\end{proposition}
\begin{proof}
This follows directly from the following chain of inequalities:
\begin{align*}
\abs{\widehat{E^\sigma_a}(\xi)} = \abs{\left(\frac{1-e^{-(a+i\xi)}}{a+i\xi}\right)^\sigma} \leq \frac{2^\sigma}{\abs{a + i \xi}^\sigma} = \frac{2^\sigma}{(a^2+\xi^2)^{\sigma/2}}\leq \frac{2^\sigma}{\abs{\xi}^\sigma},\qquad\abs{\xi} \gg 1.
\end{align*}
\end{proof}
\begin{proposition}
${E^\sigma_a}$ is in the Sobolev space $W^{s,2}(\R)$ for $s < \sigma -\frac12$.
\end{proposition}
\begin{proof}
This is implied by Proposition \ref{prop1} and the corresponding result for polynomial B-splines (cf. \cite[Section 5.1]{forster06}).
\end{proof}
\begin{proposition}\label{prop3}
${E^\sigma_a}\in C^{\lfloor\sigma\rfloor - 1}(\R)$.
\end{proposition}
\begin{proof}
The function $\xi\mapsto \frac{\xi^n}{(a^2+\xi^2)^{\sigma/2}}$ is in $L^1(\R)$ only if $n \leq \lfloor\sigma\rfloor - 1$.
\end{proof}
\section{The Interpolation Problem for Fractional Exponential B-Splines}
In order to solve the cardinal spline interpolation problem using the classical Curry-Schoenberg splines \cite{chui,schoenberg}, one constructs a fundamental cardinal spline function that is a linear bi-infinite combination of polynomial B-splines $B_n$ of fixed order $n\in \N$ which interpolates the data set $\{\delta_{m,0}: m\in \Z\}$. More precisely, one looks for a solution of the bi-infinite system
\be\label{intprob}
\sum_{k\in \Z} c_k^{(n)} B_n \left(\frac{n}{2} + m - k\right) = \delta_{m,0},\quad m\in \Z,
\ee
i.e., for a sequence $\{c_k^{(n)}: k\in \Z\}$. The left-hand side of (\ref{intprob}) defines the fundamental cardinal spline $L_n:\R\to\R$ of order $n\in \N$. A formula for $L_n$ is given in terms of its Fourier transforms by
\be\label{fundspline}
\hL_n (\xi) = \frac{\left(\hB_n (\cdot + \frac{n}{2})\right)(\xi)}{\displaystyle{\sum_{k\in \Z}}\,\left(\hB_n (\cdot + \frac{n}{2})\right)(\xi + 2\pi k)}.
\ee
Using the Euler-Frobenius polynomials associated with the B-splines $B_n$, one can show that the denominator in (\ref{fundspline}) does not vanish on the unit circle $|{z}| = 1$, where ${z} = e^{-i \xi}$. For details, see \cite{chui, schoenberg}.

One of the goals in the theory of fractional exponential B-splines is to construct a {\em fundamental cardinal exponential spline $L_a^\sigma:\R\to \R$  of real order $\sigma>1$} of the form
\be\label{compint2}
L_a^\sigma := \sum_{k\in \Z} c_k^{(\sigma)} E_a^\sigma \left(\mydot - k\right),
\ee
satisfying the interpolation problem
\be\label{complexint2}
L_a^\sigma (m) = \delta_{m,0}, \quad m\in \Z, 
\ee
for an appropriate bi-infinite sequence $\{c_k^{(\sigma)} : k\in \Z\}$ and for an appropriate $\sigma$ belonging to some nonempty subset of $\R$. 

Taking the Fourier transform of (\ref{compint2}) and (\ref{complexint2}), applying the Poisson summation formula and eliminating the expression containing the unknowns $\{c_k^{(z)}: k\in \Z\}$, a formula for $L_a^\sigma$ similar to (\ref{fundspline}) is, at first, formally obtained:
\be\label{compfundspline}
\wh{L_a^\sigma} (\xi) = \frac{\wh{E_a^\sigma} (\xi)}{\displaystyle{\sum\limits_{k\in \Z}}\,\wh{E_a^\sigma} (\xi + 2\pi k)}.
\ee
Inserting \eqref{expspline} into the above expression for $\wh{L_a^\sigma}$ and simplifying yields
\[
\wh{L_a^\sigma} (\xi) = \frac{1/(\xi + i a)^\sigma}{\displaystyle{\sum_{k\in \Z}}\, \frac{1}{[\xi + 2\pi k + i a)]^\sigma}},\quad \sigma > 1.
\]
As the denominator of (\ref{compfundspline}) is $2\pi$-periodic in $\xi$, we may assume without loss of generality that $\xi\in [0,2\pi]$. Let $q:= q(a):=\frac{\xi + i a}{2\pi}$, and note that $0\leq \Re q\leq 1$ and $\Im q > 0$. The denominator in the above expression for $\wh{L_a^\sigma}$ can then - after cancelation of the $(2\pi)^\sigma$ term - be formally rewritten in the form
\begin{align}
\sum_{k\in \Z}\, \frac{1}{(q + k)^\sigma}  &= \sum_{k=0}^\infty\, \frac{1}{(q + k)^\sigma} + \sum_{k=0}^\infty\, \frac{1}{(q - (1 + k))^\sigma}
 \nonumber\\
&= \sum_{k=0}^\infty\, \frac{1}{(q + k)^\sigma} + e^{-i \pi \sigma}\sum_{k=0}^\infty\, \frac{1}{(1 - q + k)^\sigma}
\nonumber\\
&= \zeta (\sigma,q) + e^{-i \pi \sigma}\,\zeta (\sigma, 1 -q),
\label{Zerlegung in zetas}
\end{align}
where we take the principal value of the multi-valued function $e^{-i \pi (\cdot)}$ and where $\zeta(\sigma,q)$, $q\notin\Z_0^-$, denotes the generalized zeta function \cite[Section 1.10]{erdelyi} which agrees with the Hurwitz zeta function when $\Re q > 0$, the case we are dealing with here.

For $\xi = 0$, we have $\Re q = 0$ and thus
\begin{align*}
\sum_{k=0}^\infty\, \frac{1}{\abs{q + k}^\sigma} &= \sum_{k=0}^\infty\, \frac{1}{(k^2 + a^2/4\pi^2)^{\sigma/2}} \\
&\leq \left(\frac{2\pi}{a}\right)^\sigma + \int_0^1 \frac{dx}{(x^2 + a^2/4\pi^2)^{\sigma/2}} + \int_1^\infty \frac{dx}{(x^2 + a^2/4\pi^2)^{\sigma/2}} < \infty,
\end{align*}
as $\sigma > 1$. The last integral above evaluates to 
\[
\left(\frac{a}{2\pi}\right)^{-\sigma}\left[\left(\frac{a}{4\pi}\right)\,B\left(\frac12,\frac{\sigma-1}{2}\right) - \,_2F_1\left(\frac{1}{2},\frac{\sigma }{2};\frac{3}{2};-\frac{4 \pi^2}{a^2}\right)\right],
\]
where $B$ and $_2F_1$ denote the Beta and Gauss's hypergeometric function, respectively. (See, e.g., \cite{GR}.)

Replacing in the above expression $k$ by $k+1$, one shows in a similar fashion that $\sum\limits_{k=0}^\infty\, \frac{1}{(1 - q + k)^\sigma}$ also converges absolutely. Hence, $\wh{L_a^\sigma}$ is defined and finite at $\xi = 0$.

For $\xi = 2\pi$, $q$ and $1-q$ are interchanged and we immediately obtain from the above arguments that $\wh{L_a^\sigma}$ is defined and finite at $\xi = 2\pi$. Thus, it suffices to consider $0 < \Re q < 1$.

Next, we show that the denominator in (\ref{compfundspline}) does not vanish, i.e., that $L_a^\sigma$ is well-defined for appropriately chosen $\sigma$. To this end, it suffices to find conditions on $\sigma$ such that the function
$$
Z(\sigma, q) :=\zeta (\sigma,q) + e^{-i \pi \sigma}\,\zeta (\sigma, 1 - q)
$$
has no zeros for all $\Re q\in (0,1)$ and a fixed $a >0$. 

We require the following lemma which is based on a result in \cite{spira} for the case of real $q$.

\begin{lemma}\label{lem1}
Let $q = u + i v$ where $0 < u := \frac{\xi}{2\pi} < 1$ and $v := \frac{a}{2\pi} > 0$. If \[
\sigma > \sigma_0 := \tfrac12+\sqrt{2}\sqrt{1+v^2+v^4},
\] 
then $\zeta (\sigma, q) \neq 0$.
\end{lemma}

\begin{proof}
We have that
\begin{align*}
\abs{\zeta(\sigma, q)} & \geq \frac{1}{\abs{q}^\sigma} - \sum_{k\geq 1}\frac{1}{\abs{k+q}^\sigma} \\
& >  \frac{1}{\abs{q}^\sigma} - \frac{1}{\abs{q+1}^\sigma} - \int_1^\infty \frac{dx}{\big[(u+x)^2 + v^2\big]^{\sigma/2}}.
\end{align*}
Now, for $x\geq 1$, $0<u<1$ and $v>0$, 
\[
\sqrt{(u+x)^2 + v^2} \geq \sqrt{(u+1)^2 + v^2} + \frac{(1+u)(x-1)}{\sqrt{(u+1)^2 + v^2}}
\]
as can be shown by direct computation:
\[
(u+x)^2 + v^2 - \left(\sqrt{(u+1)^2 + v^2} + \frac{(1+u)(x-1)}{\sqrt{(u+1)^2 + v^2}}\right)^2 = \frac{(x-1)^2 v^2}{(u+1)^2 + v^2}\geq 0.
\]
The above inequality shows that
\begin{align*}
\int_1^\infty \frac{dx}{\big[(u+x)^2 + v^2\big]^{\sigma/2}} & < \int_1^\infty \frac{dx}{\left(\sqrt{(u+1)^2 + v^2} + \frac{(1+u)(x-1)}{\sqrt{(u+1)^2 + v^2}}\right)^\sigma}\\
& = \frac{\left[(u+1)^2+v^2\right]^{1-\sigma/2}}{(\sigma -1 )(u+1)}.
\end{align*}
Therefore, 
\[
\abs{\zeta(\sigma, q)} > \frac{1}{\abs{q}^\sigma} - \frac{1}{|q+1|^\sigma} - \frac{\left[(u+1)^2+v^2\right]^{1-\sigma/2}}{(\sigma -1 )(u+1)},
\]
and the right-hand side of this inequality is strictly positive if
\be\label{1}
\frac{1}{\abs{q}^\sigma} > \frac{1}{|q+1|^\sigma} + \frac{\left[(u+1)^2+v^2\right]^{1-\sigma/2}}{(\sigma -1 )(u+1)}.
\ee

Replacing $q$ by $u + i v$ and simplifying shows that inequality \eqref{1} is equivalent to
\be\label{12}
\left(1+\frac{2u+1}{u^2+v^2}\right)^{\sigma/2} > 1 + \frac{(u+1)^2+v^2}{(\sigma -1 )(u+1)}.
\ee
Employing the Bernoulli inequality to the expression on the left-hand side of \eqref{12}, yields
\[
\left(1+\frac{2u+1}{u^2+v^2}\right)^{\sigma/2}  \geq 1 + \frac{\sigma}{2}\frac{2u+1}{u^2+v^2},
\]
which implies that \eqref{12} holds if
\[
\frac{\sigma}{2}\cdot\frac{2u+1}{u^2+v^2} > \frac{(u+1)^2+v^2}{(\sigma -1 )(u+1)},
\]
or, equivalently,
\[
\sigma\, (\sigma -1) > \frac{2(u^2+v^2)[(u+1)^2+v^2]}{(1+u)(1+2u)}.
\]
Performing the polynomial division on the right-hand side of the above inequality produces
\[
\sigma\, (\sigma -1) > 2v^2+u^2+\tfrac12 u - \tfrac14+\frac{\frac14+2v^4+\frac14 u -2 u v^2}{(1+u)(1+2u)}
\]
and this inequality holds if 
\be\label{3}
\sigma\, (\sigma -1) > 2v^2 + 2v^4 -\tfrac74,
\ee
where we used the fact that $0<u<1$.

Thus, inequality \eqref{3} holds if 
\[
\sigma > \sigma_0 := \tfrac12+\sqrt{2}\sqrt{1+v^2+v^4}.\qedhere
\]
\end{proof}

\begin{theorem}
The function $Z(\sigma, q) = \zeta (\sigma,q) + e^{-i \pi \sigma}\,\zeta (\sigma, 1 - q)$ with $q = \frac{1}{2\pi}(\xi + i\,{a})$ has no zeros provided 
\be\label{sigma}
\sigma \geq \sigma_0 = \tfrac12+\sqrt{2}\sqrt{1+\frac{a^2}{4\pi^2}+\frac{a^4}{16\pi^4}}
\ee
and
\be\label{2}
\frac{\pi}{2}(\sigma -1) + \Arg \left(\zeta (\sigma, \tfrac12 -i\,\tfrac{a}{2\pi})\right) \notin \pi\N.
\ee
\end{theorem}
\begin{proof}
For a given $q$, we consider three cases: (I) $0 < \re q < \frac12$, (II) $\frac12 < \re q < 1$, and (III) $\re q = \frac12$.

To this end, fix $a > 0$ and choose $\sigma > \sigma_0$. Note that the above argument employed to derive $\sigma_0$ also applies to the case of $q$ being replaced by$1-q = 1-u - iv$ and yields the same value. Thus, by Lemma \ref{lem1}, $\zeta(\sigma, q) \neq 0$ and $\zeta(\sigma, 1-q) \neq 0$.
\nl
Case I: If $0 < \re q < \frac12$ then $\abs{k+q} < \abs{k+1-q}$, for all $k\in \N_0$. Therefore,
\begin{align*}
\abs{e^{-i\,\pi\sigma}\,\zeta(\sigma, 1 -q)} & < \sum_{k=0}^\infty \frac{1}{\abs{k+q}^\sigma} = \abs{\zeta(\sigma, q)}.
\end{align*}
Similarly, one obtains in Case II with $\frac12 < \re q < 1$ that
\begin{align*}
\abs{\zeta(\sigma, q)} & < \abs{e^{-i\,\pi\sigma}\,\zeta(\sigma, 1 -q)}.
\end{align*}
Hence, $\abs{Z(\sigma, q)} \geq \abs{\abs{\zeta(\sigma, q)}-\abs{\zeta(\sigma, 1-q)}} > 0$, for $\re q \neq \frac12$.

In Case III with $\re q = \frac12$ and $\sigma$ satisfying \eqref{sigma}, we set $q^* := \frac12  + i\,\frac{a}{2\pi}$ and observe that 
\[
\zeta(\sigma, q^*) = \sum\limits_{k=0}^\infty \frac{1}{(k+\frac12 + i\,\frac{a}{2\pi})^\sigma}
\]
and 
\[
\zeta(\sigma, 1- q^*) = \sum\limits_{k=0}^\infty \frac{1}{(k+\frac12 - i\,\frac{a}{2\pi})^\sigma}.
\]
Hence, $\zeta(\sigma, 1- q^*) = \overline{\zeta(\sigma, q^*)}$ and therefore
\[
Z(\sigma, q^*) = \zeta(\sigma, q^*) + e^{-i\,\pi\sigma} \overline{\zeta(\sigma, q^*)} = 
\zeta(\sigma, q^*) \left(1 + e^{-i\,\pi\sigma}\, \frac{\overline{\zeta(\sigma, q^*)}}{ \zeta(\sigma, q^*)}\right).
\]
Setting for simplicity $\zeta^* := \zeta(\sigma, q^*)$, the expression in parentheses becomes zero if
\[
1 + e^{-i\,\pi\sigma}\,\left( \frac{\overline{\zeta^*}}{ \zeta^*} \right)= 0
\]
or, equivalently, as $\frac{\overline{\zeta^*}}{ \zeta^*} = \exp (-2 i \arg \zeta^*)$,
\[
\exp(-i\,\pi\sigma -2 i \arg \zeta^*) = \exp(i\,\arg (-1)).
\]
Using the principal values of $\arg$, $\Arg$, this latter equation can be rewritten as
\[
\frac{\pi}{2}\,\sigma + \Arg(\zeta^*) = (2m + 1)\,\frac{\pi}{2}, \quad m\in \Z.
\]
Note that $\sigma \geq 1 + \sqrt{1 + \frac{a^2}{4\pi^2}} > 2$ and that $\Arg(\zeta^*)\in (-\pi, \pi]$ (taking the negative real axis as a branch cut) and therefore, we need to impose condition \eqref{2} to ensure that $Z(\sigma, q^*) \neq 0$ as $\zeta(\sigma, q^*) \neq 0$.
\end{proof}

\begin{remark}
As $-\pi < \Arg z \leq\pi$, condition \eqref{2} can for fixed $a>0$ have at most one solution. 
\end{remark}

\begin{remark}
Note that if $q\in \R$, i.e., $ a = 0$, we obtain the conditions derived in \cite{FM} for polynomial B-splines of fractional order. 
\end{remark}

\begin{definition}
We call real orders $\sigma$ that fulfill conditions \eqref{sigma} and \eqref{2} for a fixed $a > 0$ admissible.
\end{definition}

\begin{example}
Let $\sigma := \sqrt{6}$ and $a:=2$. Hence, $\sqrt{6} > \sigma_0 \approx 1.99103$ and condition \eqref{sigma} holds. A numerical evaluation of $\zeta( \sqrt{6}, \frac12+\frac{i}{\pi})$ using Mathematica's HurwitzZeta function produces the value $\zeta^* = 1.19269 - i\,3.76542$. The principal value of $\arg\zeta^*$ is therefore $\Arg\zeta^* = -1.26405$. As $\frac{\pi}{2}(\sigma -1) + \Arg\zeta^* = 1.01281 \notin \pi\N$, the second condition \eqref{2} is also satisfied. Thus, $\sigma := \sqrt{6}$ is an admissible real order. By the continuous dependence of conditions \eqref{sigma} and \eqref{2} on $\sigma$, there exist therefore uncountably many admissible $\sigma$.
\end{example}

\begin{example}
Selecting $a := 2$, the value $\sigma \approx 4.68126$ is not admissible as condition \eqref{sigma} is satisfied but the left-hand side of \eqref{2} yields the value $\pi$. (See Figure \ref{fig0} below.)
\begin{figure}[h!]
\begin{center}
\includegraphics[width=5cm, height= 3cm]{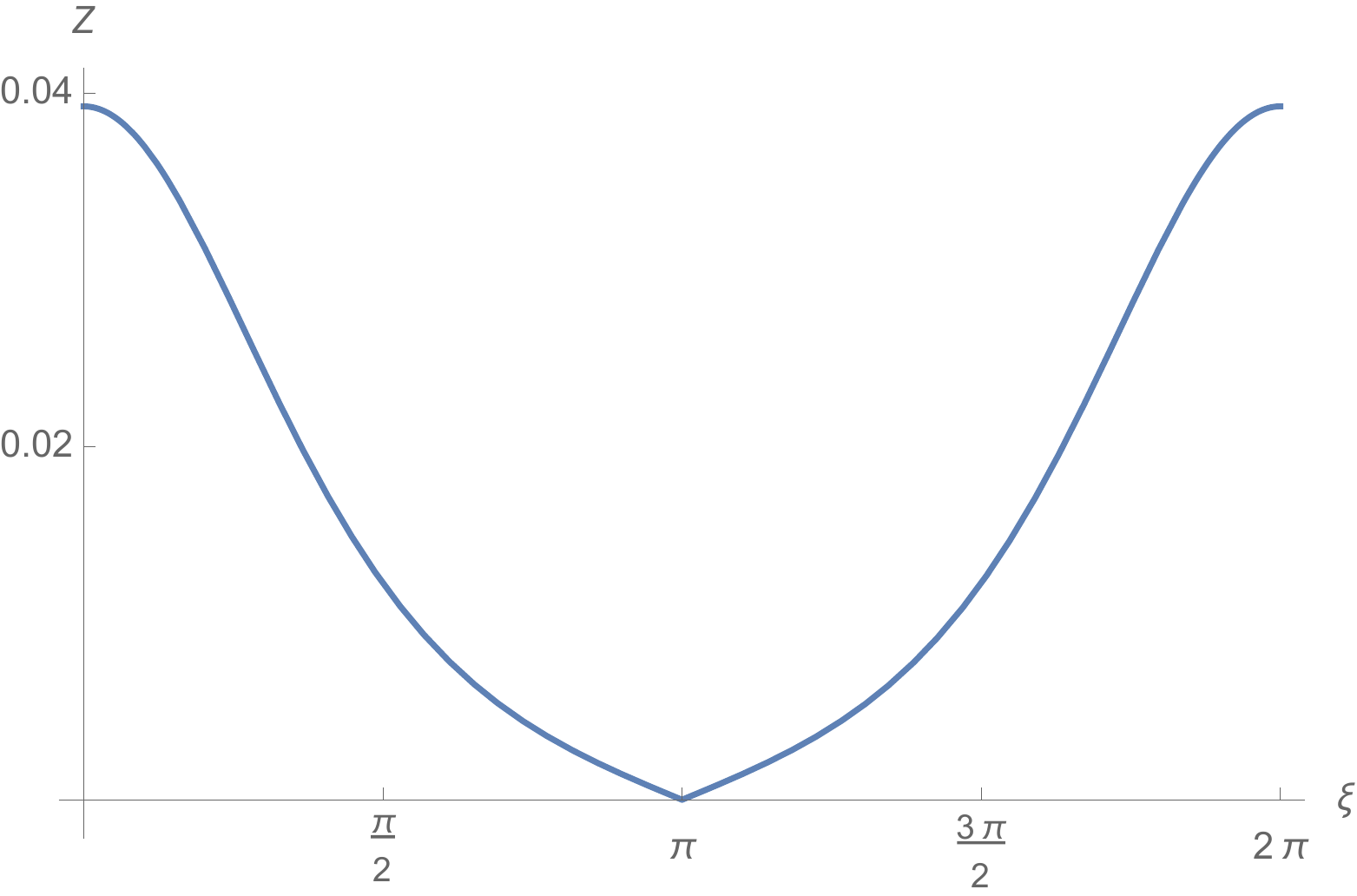}
\caption{Example of a non-admissible $\sigma$.}\label{fig0}
\end{center}
\end{figure}
\end{example}

We finally arrive at one of the main results.

\begin{theorem}
Suppose that ${E}_{a}^\sigma$ is an exponential B-spline of admissible real order $\sigma$.
Then
\be\label{L}
{L}_a^\sigma (x) := \frac{1}{2\pi}\,\int\limits_{\R} \frac{((\xi + i a)/2\pi)^{-\sigma}\,e^{i \xi x}\,d\xi}{\zeta (\sigma, (\xi + i\,a)/2\pi) + e^{-i \pi \sigma}\zeta (\sigma, 1 - (\xi + i\,a)/2\pi)}
\ee
is a fundamental exponential interpolating spline of real order $\sigma$ in the sense that
\[
{L}_a^\sigma ({m}) = \delta_{{m},0}, \quad \mbox{for all }{m}\in \Z.
\]
The Fourier inverse in \eqref{L} holds in the $L^1$ and $L^2$ sense.
\end{theorem}

Let
\[
h(\xi, a, \sigma) := \frac{((\xi + i a)/2\pi)^{-\sigma}}{\zeta (\sigma, (\xi + i\,a)/2\pi) + e^{-i \pi \sigma}\zeta (\sigma, 1 - (\xi + i\,a)/2\pi)}.
\]
\nl
The following figures show $\abs{h}$, $\Re h$, and $\Im h$ as functions of fixed 
$a := 2$ and varying $\sigma\in \{2.5, 2.75, 3, 3.5\}$, and for fixed $\sigma:= \sqrt{6}$ and varying $a\in \{2, 3, 4, 5\}$.

\begin{figure}[h!]
\begin{center}
\includegraphics[width=6cm, height= 4cm]{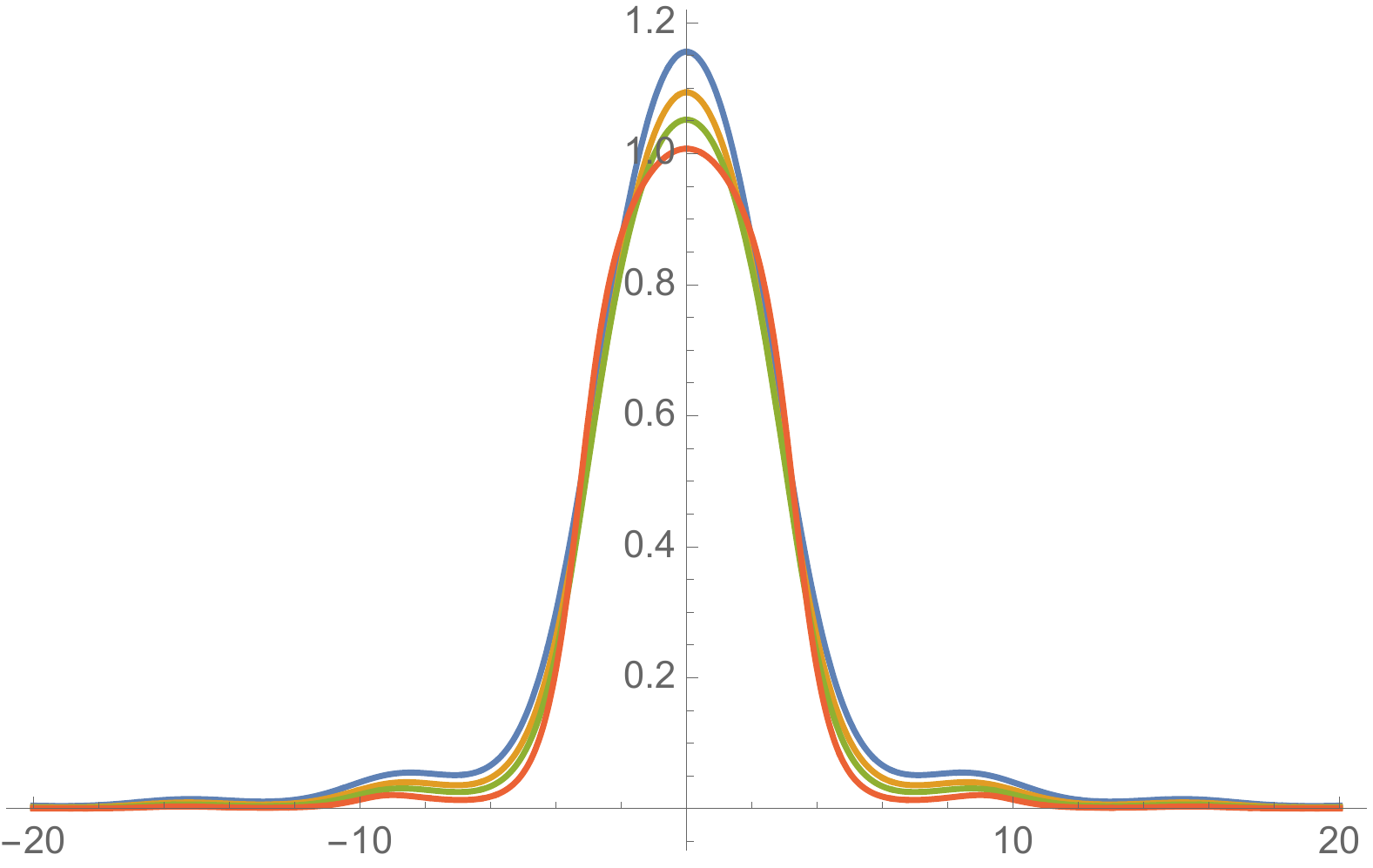}\qquad\includegraphics[width=6cm, height= 4cm]{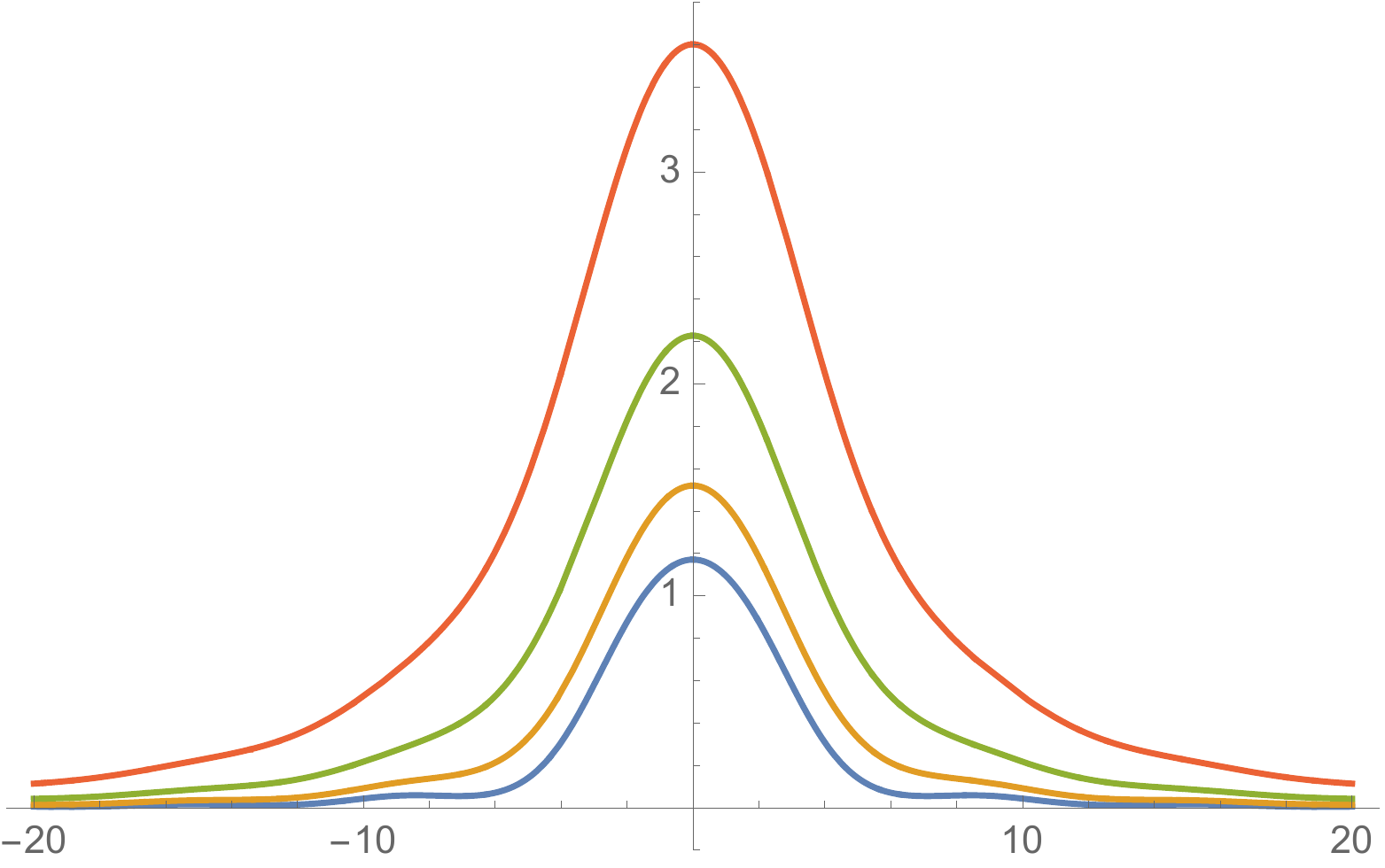}
\caption{$\abs{h(\xi, a, \sigma)}$ for fixed $a$ and varying $\sigma$ (left) and for fixed $\sigma$ and varying $a$ (right).}
\end{center}
\end{figure}

\begin{figure}[h!]
\begin{center}
\includegraphics[width=6cm, height= 4cm]{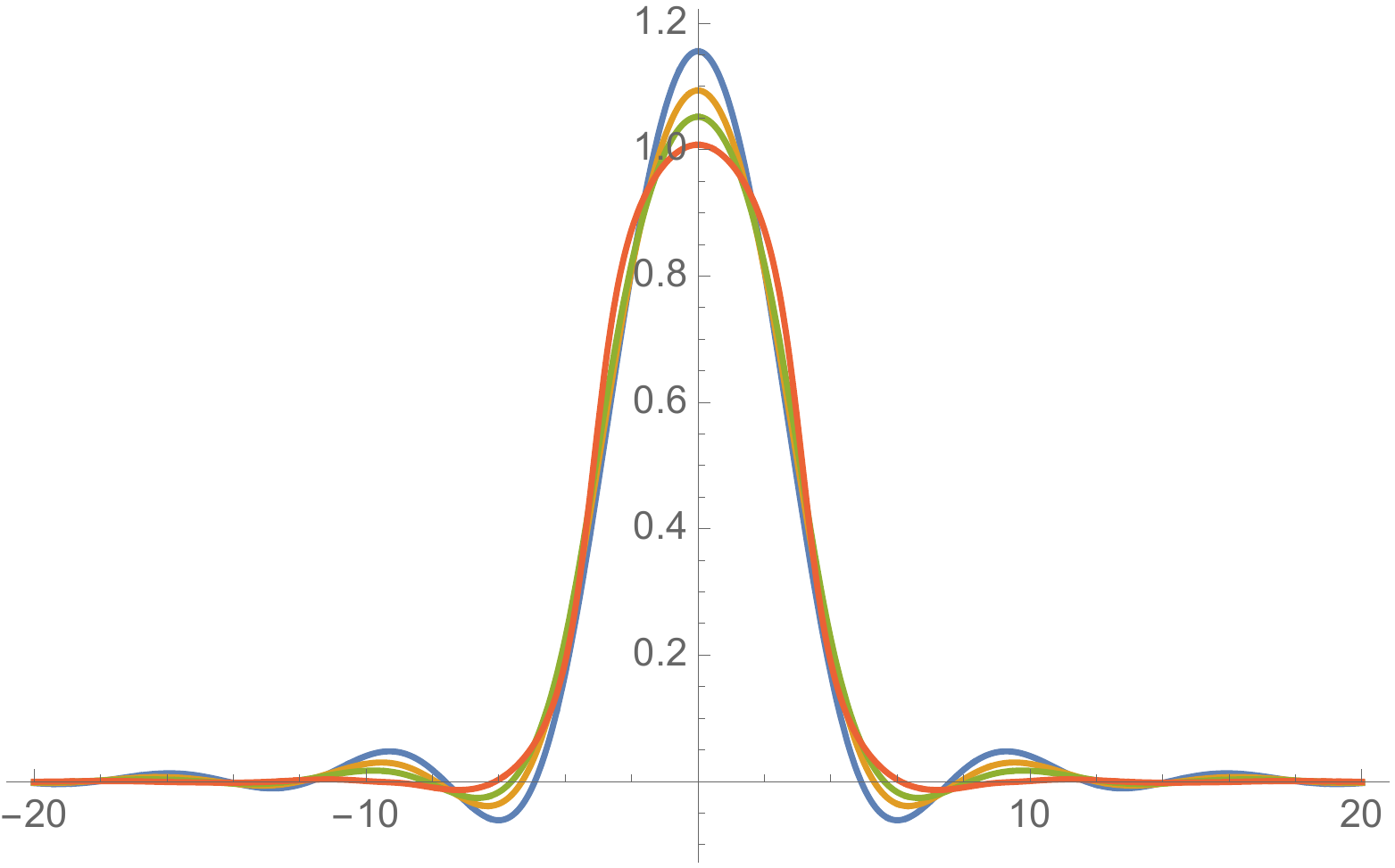}\qquad\includegraphics[width=6cm, height= 4cm]{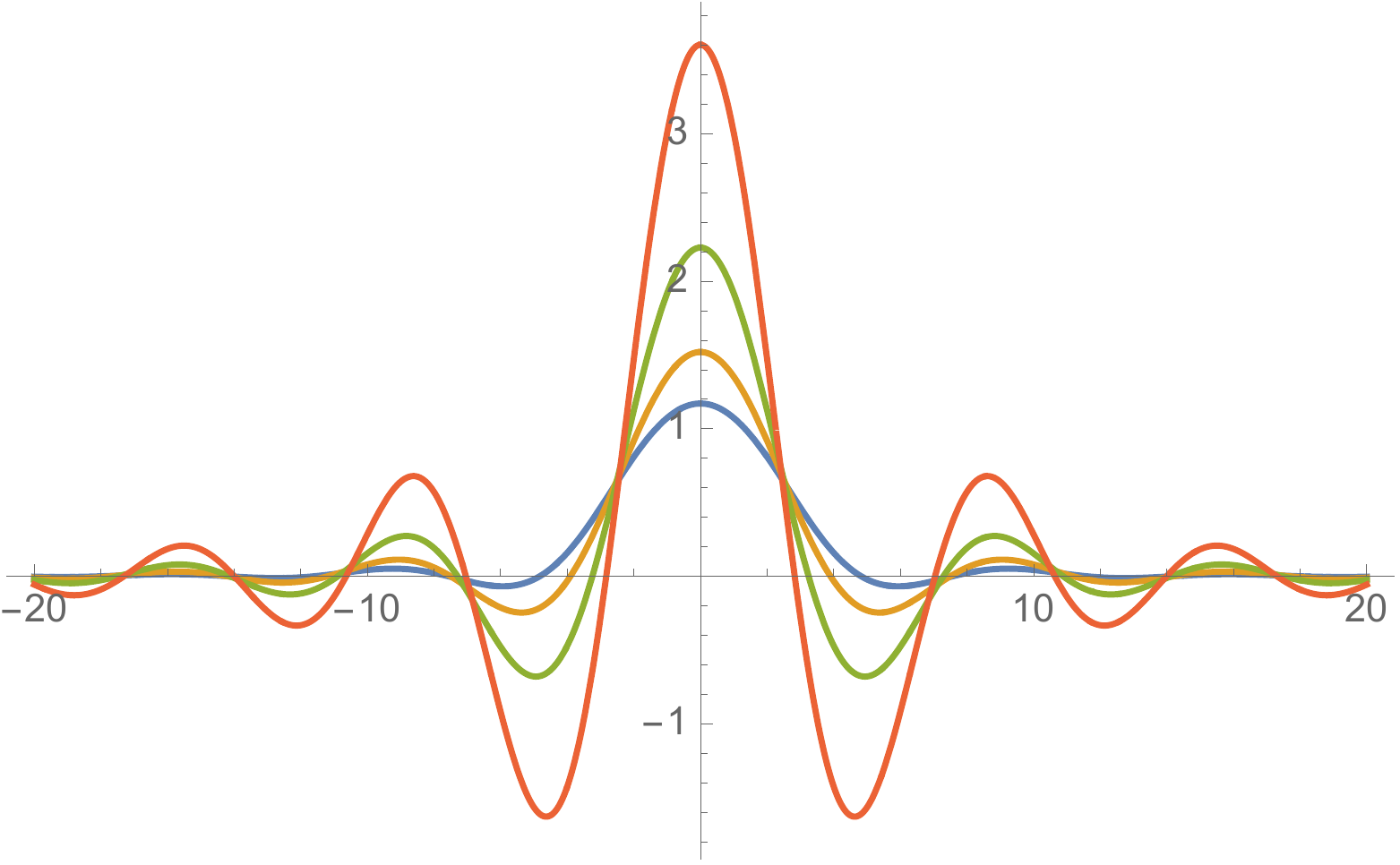}
\caption{$\Re h(\xi, a, \sigma)$ for fixed $a$ and varying $\sigma$ (left) and for fixed $\sigma$ and varying $a$ (right).}
\end{center}
\end{figure}

\begin{figure}[h!]
\begin{center}
\includegraphics[width=6cm, height= 4cm]{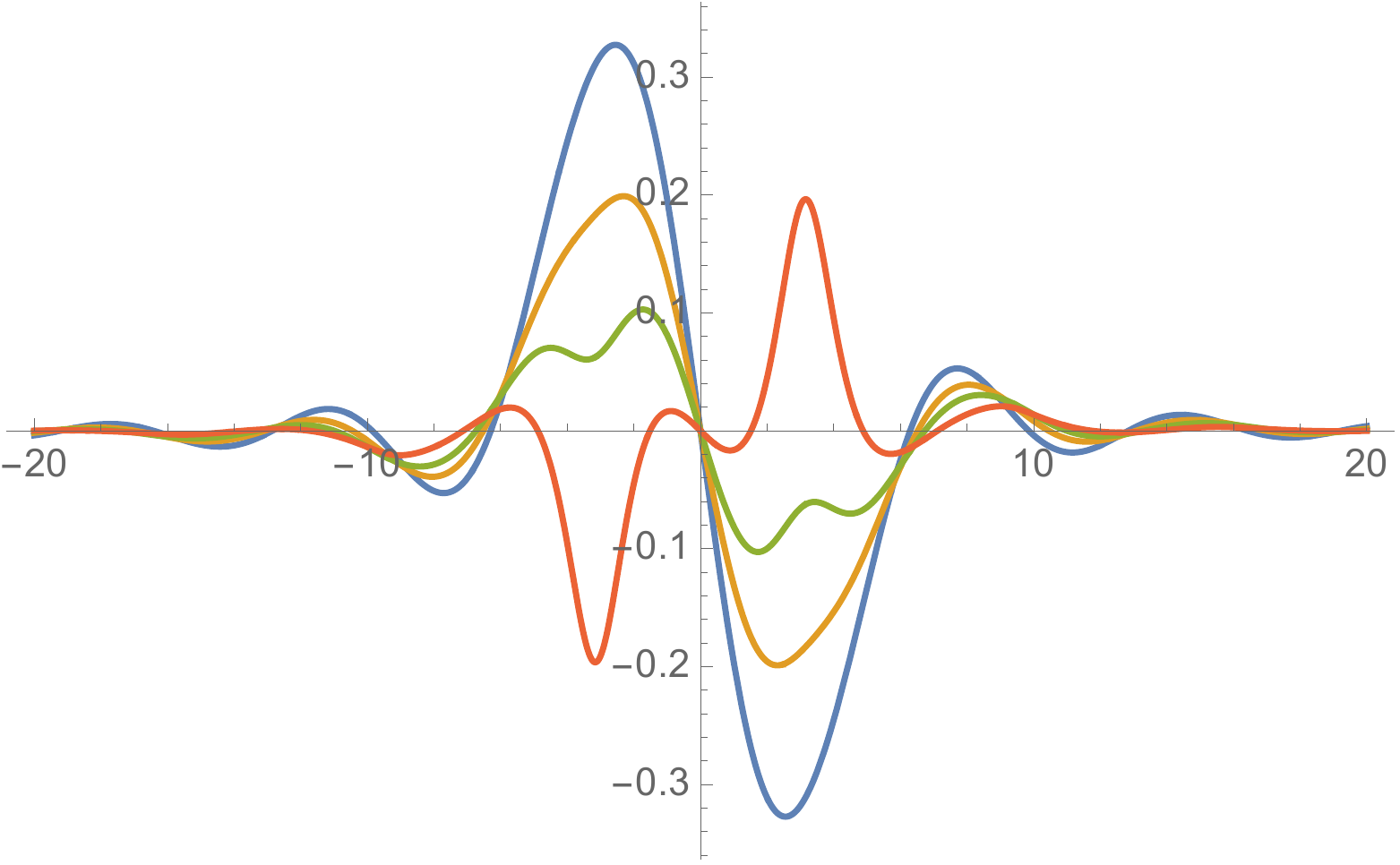}\qquad\includegraphics[width=6cm, height= 4cm]{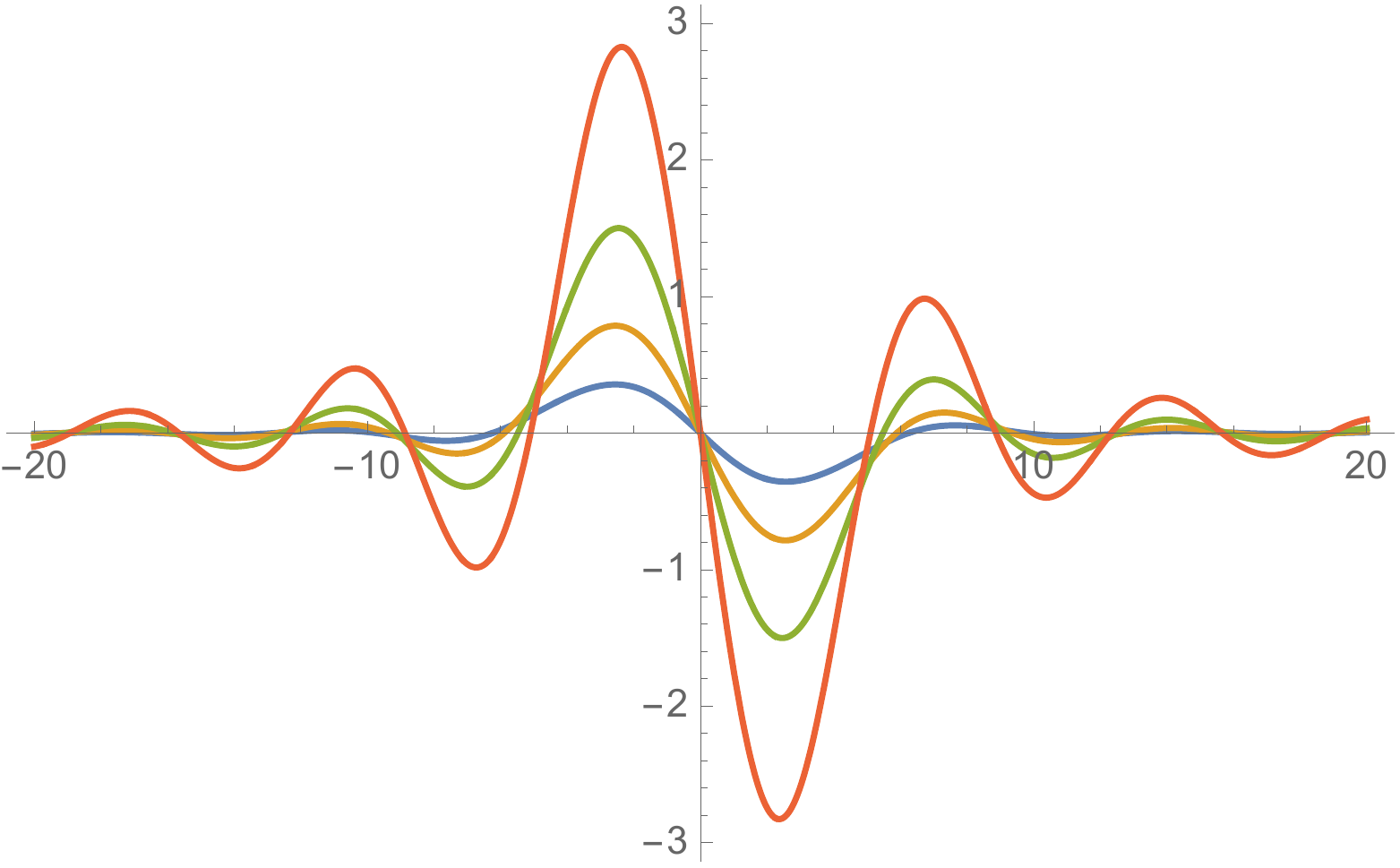}
\caption{$\Im h(\xi, a, \sigma)$ for fixed $a$ and varying $\sigma$ (left) and for fixed $\sigma$ and varying $a$ (right).}
\end{center}
\end{figure}

\begin{example}
We choose again $a:=2$ and $\sigma \in\{\sqrt{6}, 3.5, 4.25\}$. Figure \ref{fig4} below displays the graph of the fundamental exponential interpolating spline $L_2^{\sigma}$.
\end{example}
\begin{figure}[h!]
\begin{center}
\includegraphics[width=10cm, height= 5cm]{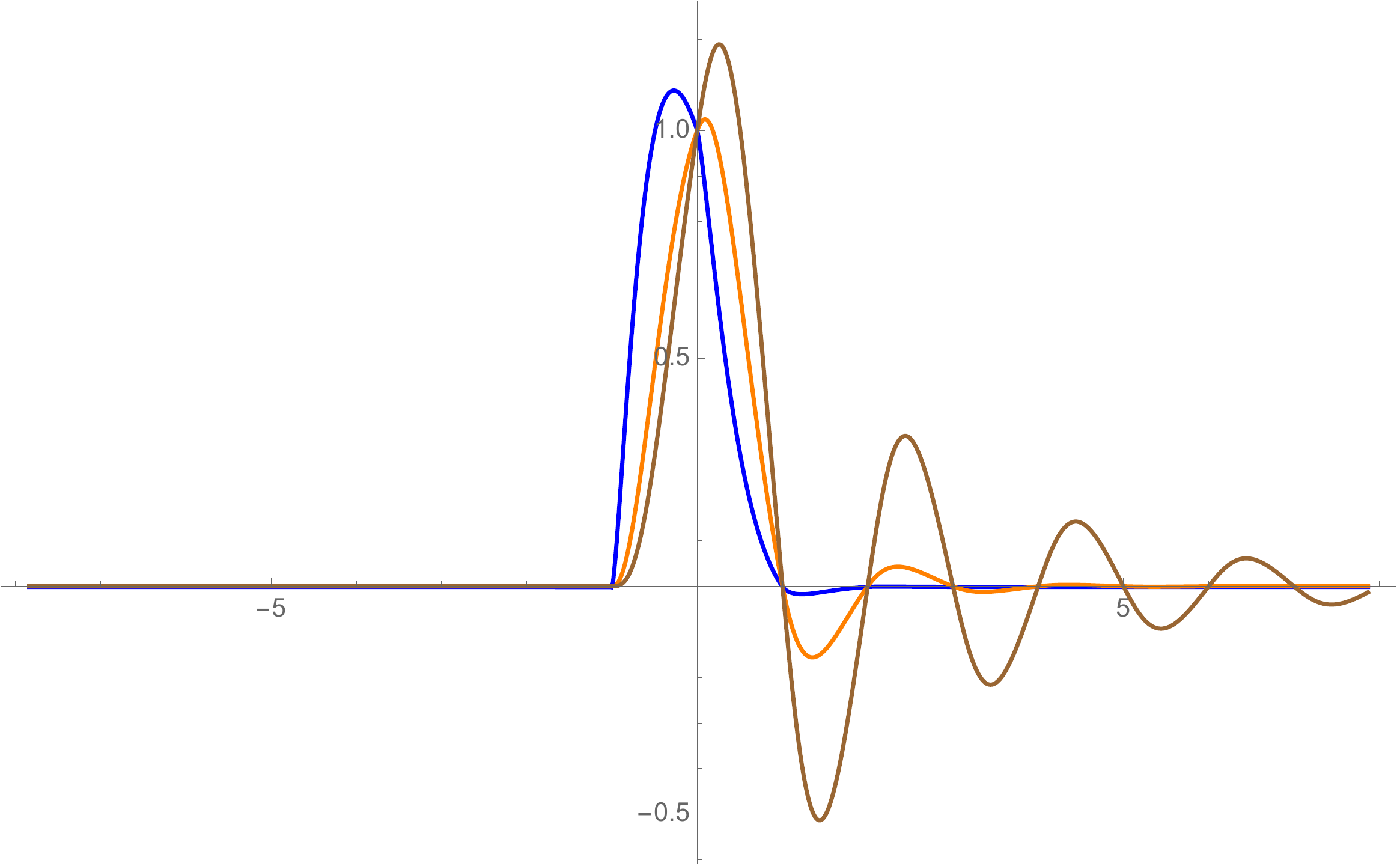}\
\caption{The fundamental exponential interpolating splines $L_2^{\sigma}$ with $\sigma \in\{\sqrt{6}, 3.5, 4.25\}$.}\label{fig4}
\end{center}
\end{figure}
\begin{proposition}
The coefficients $c_k^{(\sigma)}$ in Eqn. \eqref{compint2} decay like
\[
\abs{c_k^{(\sigma)}} \leq C_\sigma\,\abs{k}^{\lfloor\sigma\rfloor -1},
\]
for some positive constant $C_\sigma$. Therefore, the fundamental exponential spline $L_a^\sigma$ with admissible $\sigma$ satisfies the pointwise estimate
\[
\abs{L_a^\sigma (x)} \leq M_q\,\abs{x}^{-\lfloor\sigma\rfloor}, \quad x\in\R, 
\]
where $M_\sigma$ denotes a positive constant.
\end{proposition}
\begin{proof}
Eqns. \eqref{compint2} and \eqref{compfundspline} together with the Poisson summation formula yield
\[
\sum_{k\in\Z} c_k^{(\sigma)} w^k = \frac{1}{\sum\limits_{k\in\Z} E_a^\sigma (k)\, w^k} =: \vartheta_a^\sigma (w), \quad w = e^{i \xi}.
\]
The function $\vartheta_a^\sigma (w)$ has no zeros on the unit circle $\abs{w}=1$ provided $\sigma$ is admissible. Proposition \ref{prop3} implies that the Fourier coefficients of $\vartheta_a^\sigma (w)$ satisfy
\[
\abs{c_k^{(\sigma)}} \leq C_\sigma\,\abs{k}^{\lfloor\sigma\rfloor -1},
\]
for some positive constant $C_\sigma$.

Noting that $\supp E_a^\sigma (\mydot - k) = [k,\infty)$, $k\in \Z$, we thus obtain using Eqn. \eqref{compint2}
\[
\abs{L_a^\sigma (x)} = \abs{\sum_{k=-\infty}^{\lfloor x\rfloor} c_k^{(\sigma)}\,E_a^\sigma (x-k)} \leq K_\sigma \sum_{k=-\infty}^{\lfloor x \rfloor} \abs{k}^{-\lfloor\sigma\rfloor} \leq M_q\,\abs{x}^{-\lfloor\sigma\rfloor}.
\]
Here, we used the boundedness of $E_a^\sigma$ on $\R$. (Cf. for instance \cite[Proposition 4.5]{m14}.) 
\end{proof}
\section{A Sampling Theorem}
In this section, we derive a sampling theorem for the fundamental cardinal exponential spline $L_a^\sigma$, where $\sigma$ satisfies conditions \eqref{sigma} and \eqref{2}. For this purpose, we employ the following version of Kramer's lemma \cite{kramer} which appears in \cite{garcia}. We summarize those properties that are relevant for our needs.
\begin{theorem}\label{gensamp}
Let $\emptyset\neq I$, $M\subseteq\R$ and let $\{\varphi_k: k\in\Z\}$ be an orthonormal basis of $L^2(I)$. Suppose that $\{S_k: k\in \Z\}$ is a sequence of functions $S_k: M\to\C$ and $\boldsymbol{t} := \{t_k\in \R: k\in \Z\}$ a numerical sequence in $M$ satisfying the conditions
\begin{enumerate}
\item[C1.]	$S_k(t_l) = a_k \delta_{kl}$, $(k,l)\in \Z\times \Z$, where $a_k\neq 0$;
\item[C2.]	${\sum\limits_{k\in \Z}}\,\vert S_k(t)\vert^2 < \infty$, for each $t\in \xi$.
\end{enumerate}
Define a function $K:I\times M \to \C$ by
\[
K(x,t) := \sum_{k\in \Z} S_k (t) \overline{\varphi_k} (x),
\]
and a linear integral transform $\mcK$ on $L^2 (I)$ by
\[
(\mcK f)(t) := \int_I f(x) K(x,t) \, dx.
\]
Then $\mcK$ is well-defined and injective. Furthermore, if the range of $\mcK$ is denoted by
\[
\mcH := \left\{g:\R\to\C : g = \mcK f, \,f\in L^2(I)\right\},
\]
then
\begin{enumerate}
\item[(i)]	$(\mcH, \inn{\cdot}{\cdot}_\mcH)$ is a Hilbert space isometrically isomorphic to $L^2(I)$,  $\mcH \cong L^2(I)$, when endowed with the inner product
\[
\inn{F}{G}_\mcH := \inn{f}{g}_{L^2(I)},
\]
where $F := \mcK f$ and $G = \mcK g$.
\item[(ii)] $\{S_k: k\in \Z\}$ is an orthonormal basis for $\mcH$.
\item[(iii)] Each function $f\in \mcH$ can be recovered from its samples on the sequence $\{t_k: k\in \Z\}$ via the formula
\[
f(t) = \sum_{k\in \Z} f(t_k)\,\frac{S_k (t)}{a_k}.
\]
The above series converges absolutely and uniformly on subsets of $\R$ where the kernel $K(\,\cdot\,, t)$ is bounded in $L^2(I)$.
\end{enumerate}
\end{theorem}
\begin{proof}
For the proof and further details, we refer to \cite{garcia}.
\end{proof}

For our purposes, we choose $M := \R$, $\boldsymbol{t} := \Z$, $a_k = 1$ for all $k\in\Z$, and for the interpolating function $S_{k} = L_a^{\sigma}(\cdot - k)$ with $\sigma$ satisfying conditions \eqref{sigma} and \eqref{2}. Then Theorem \ref{gensamp} implies the next result.

\begin{theorem}\label{S Abtastsatz}
Let $\emptyset\neq I\subseteq\R$ and let $\{\varphi_k: k\in\Z\}$ be an orthonormal basis of $L^2(I)$. Let $L_a^\sigma$ denote the fundamental cardinal spline of admissible real order $\sigma$. Then the following holds:
\begin{enumerate}
\item[(i)]	The family $\{L_a^\sigma (\cdot - k): k\in \Z\}$ is an orthonormal basis of the Hilbert space $(\mcH, \inn{\cdot}{\cdot}_\mcH)$, where $\mcH = \mcK (L^2(I))$ and $\mcK$ is the injective integral operator
\[
\mcK f = \sum_{k\in\Z} \inn{f}{\varphi_k}_{L^2(I)}\, L_a^\sigma (\cdot - k), \quad f\in L^2(I).
\]
\item[(ii)]	Every function $f\in \mcH \cong L^2(I)$ can be recovered from its samples on the integers via 
\begin{equation}
f = \sum_{k\in \Z} f(k) L_a^\sigma (\cdot - k),
\label{eq Kramer Abtastreihe}
\end{equation}
where the above series converges absolutely and uniformly on all subsets of $\R$.
\end{enumerate}
\end{theorem}
\begin{proof}
Conditions C1. and C2. for $S_{k }= L_{a}^\sigma(\cdot -k)$, $k\in\Z$, in Theorem \ref{gensamp} are readily verified. Since the unfiltered splines $\{E_a^\sigma (\cdot - k): k\in \Z\}$ already form a Riesz basis of the $L^2$-closure of their span \cite{m14}, $\|K(\cdot, t)\|_{L^2(I)}$ is bounded on $\R$.
\end{proof}

Finally, we consider two examples illustrating the above theorem. These examples can also be found in \cite{FM} in case one deals with cardinal polynomial B-splines of fractional order. 
\begin{example}
Consider $L^2[\,0,2\pi\,]$ with orthonormal basis $\{\exp(ik\,(\cdot))\}_{k\in\Z}$. Then 
$$
K(x,t) = \sum_{k\in\Z} L_{a}^\sigma(t-k) \exp(-ikx)
$$
and
\begin{eqnarray*}
\mcK f(t) & = & \int_{0}^{2\pi} f(x) \sum_{k\in\Z} L_{a}^\sigma(t-k) \exp(-ikx)\, dx 
\\
& = & \sum_{k\in\Z} \int_{0}^{2 \pi} f(x) \exp(-ikx)\, dx \, L_{a}^\sigma(t-k) 
\\
& =& 2 \pi \sum_{k\in\Z} \widehat{f}(k) L_{a}^\sigma(t-k).
\end{eqnarray*}
This equation holds in $L^2$-norm and we applied the Lebesgue dominated convergence theorem. Thus, $\mcK f$ interpolates the sequence of Fourier coefficients $\{ \widehat{f}(k)\}_{k\in\Z}$ on $\R$ with shifts of the fundamental cardinal spline $L_{a}^\sigma$ of real order $\sigma$. 

Moreover, if $f\in \mcH = \mcK(L^2[\,0,2\pi\,])\cong L^2[\,0,2\pi\,]$, then, by Theorem \ref{S Abtastsatz}, $f$ can be reconstructed from its samples by the similar series
$$
f = \sum_{k\in\Z} f(k) L_{a}^\sigma(\cdot -k),
$$
which converges absolutely and uniformly on all subsets of $\R$.
\end{example}

\begin{example}
Consider $L^2(\R)$ endowed with the (orthonormal) Hermite basis defined by
$$
\varphi_{k}(x) = \frac{(-1)^k}{k!} \exp\left(\frac{x^2}{2}\right) \left(\frac{d}{dx}\right)^k \exp(-x^2), \quad k\in\N_{0}.
$$
Then 
$$
K(x,t) = \sum_{k\in\Z} L_{z}(t-k) \varphi_{p(k)}(x),
$$
where $p: \N_{0}\to \Z$ maps the natural numbers bijectively to the integers.
An application of the Lebesgue dominated convergence theorem yields 
\begin{eqnarray*}
\mcK f (t) & = & \int_{\R} f(x) \sum_{k\in\Z} L_{z}(t-k) \varphi_{p(k)}(x) \, dx
\\
& = & \sum_{k\in\Z} \int_{\R} f(x) \varphi_{p(k)}(x)\, dx\,  L_{z}(t-k).
\end{eqnarray*}
The integral represents the coefficients of $f$ in the orthonormal basis $\{\varphi_{k}\}_{k\in\N_{0}}$.
Again by Theorem \ref{S Abtastsatz}, all functions $f \in  \mcH = \mcK(L^2(\R))\cong L^2(\R)$ can be reconstructed from its samples at the integers via the series (\ref{eq Kramer Abtastreihe}).
\end{example}

\bibliographystyle{plain}
\bibliography{Interpolation_and_sampling}

\end{document}